\documentclass[12pt,reqno]{amsart}
\usepackage{amsfonts,amstext}
\usepackage{tikz}
\usepackage{amsmath,amstext,amssymb,amscd}
\usepackage{multicol}
\usepackage{calligra}
\usepackage{amsfonts}
\usepackage{amsthm}
\usepackage{amsmath}
\usepackage{amscd}
\usepackage[latin2]{inputenc}
\usepackage{t1enc}
\usepackage[mathscr]{eucal}
\usepackage{indentfirst}
\usepackage{graphicx}
\usepackage{graphics}
\usepackage{pict2e}
\usepackage{epic}
\numberwithin{equation}{section}
\usepackage[margin=2.9cm]{geometry}
\usepackage{epstopdf} 

\newtheorem{theorem}{Theorem}
\newtheorem{lemma}[theorem]{Lemma}
\newtheorem{proposition}[theorem]{Proposition}
\newtheorem{conjecture}[theorem]{Conjecture}
\newtheorem{corollary}[theorem]{Corollary}
\newtheorem{question}[theorem]{Question}

\theoremstyle{definition}
\newtheorem{example}[theorem]{Example}
\newtheorem{remark}[theorem]{Remark}

\DeclareMathAlphabet{\mathcalligra}{T1}{calligra}{m}{n}

\title [Supersolvable line arrangements]
{Real and complex supersolvable line arrangements in the projective plane}
\author[Krishna Hanumanthu]{Krishna Hanumanthu}
\address{Chennai Mathematical Institute, H1 SIPCOT IT Park, Siruseri, Kelambakkam 603103, India}
\email{krishna@cmi.ac.in}
\author[Brian Harbourne]{Brian Harbourne}
\address{Department of Mathematics, University of Nebraska-Lincoln, Lincoln, NE 68588, USA}
\email{brianharbourne@unl.edu}
\thanks{The first author was partially supported by a grant from Infosys
        Foundation and by DST SERB MATRICS grant MTR/2017/000243.
        The second author was partially supported by Simons Foundation grant \#524858.
        We also thank \c{S}.\ Toh\v{a}neanu for some helpful comments.}
\begin{document}

\begin{abstract}
We study supersolvable line arrangements in ${\mathbb P}^2$ over the reals and over the complex numbers,
as the first step toward a combinatorial classification. Our main results show that
a nontrivial (i.e., not a pencil or near pencil) complex line arrangement cannot have more than 4 modular points, and 
if all of the crossing points of a complex line arrangement have multiplicity 3 or 4, then 
the arrangement must have 0 modular points
(i.e., it cannot be supersolvable). 
This provides at least a little evidence for
our conjecture that every nontrivial complex
supersolvable line arrangement has at least one point of multiplicity 2,
which in turn is a step toward the much stronger conjecture of 
Anzis and Toh\v{a}neanu that every nontrivial complex 
supersolvable line arrangement with $s$ lines has at least $s/2$ points of multiplicity 2.
\end{abstract}

\date{July 15, 2019}
\maketitle

%\section{Introduction}

\section{Introduction}
Line arrangements have provided useful insight in studying a range of recent problems in algebraic geometry.
They have played a fundamental role in studying the containment problem (see \cite{DST, DHNSST}),
for the bounded negativity problem and H-constants \cite{BDHHPS},
and for unexpected curves \cite{CHMN, DMO}. The supersolvable arrangements
are a particularly tractable subclass of line arrangements which have played a role
in the study of unexpected curves \cite{CHMN, DMO}. Understanding them better
should make them even more useful. Thus the goal of the present paper is to
pin down as much as currently possible properties of real and complex 
supersolvable line arrangements.

A line arrangement is simply a finite set of $s>1$ distinct lines ${\mathcal L}=\{L_1,\dots, L_s\}$ in the projective plane.
A {\em modular point} for ${\mathcal L}$ is a {\em crossing point} $p$ (i.e., a point where two (or more) of the lines meet), 
with the additional property
that whenever $q$ is any other crossing point, then the line through $p$ and $q$ is $L_i$ 
for some $i$. Then we say ${\mathcal L}$ is {\em supersolvable} if it has a modular point (see Figure \ref{FigSS}).

If the $s$ lines of ${\mathcal L}$ are concurrent (i.e., all meet at a point), then ${\mathcal L}$ is supersolvable.
Such an arrangement is called a {\em pencil}. If ${\mathcal L}$ consists of $s$ lines, exactly 
$s-1$ of which are concurrent, it is called a {\em near pencil}; near pencils are also supersolvable, since every 
crossing point for a near pencil is modular. Removing any line,
different from the line through two white dots, from the arrangement shown in Figure \ref{FigSS}
results in a near pencil.

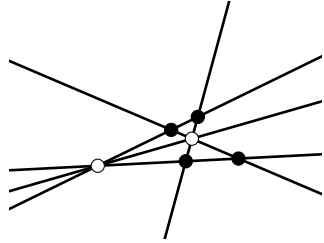
\begin{figure}
\begin{tikzpicture}[line cap=round,line join=round,x=0.5cm,y=0.5cm]
\clip(-4.3,0) rectangle (4,6.3);
\draw [line width=1.pt,domain=-4.3:7.3] plot(\x,{(-7.6824-1.3*\x)/-2.66});
\draw [line width=1.pt,domain=-4.3:7.3] plot(\x,{(-0.1872-1.18*\x)/-0.32});
\draw [line width=1.pt,domain=-4.3:7.3] plot(\x,{(--4.7724--0.12*\x)/2.34});
\draw [line width=1.pt,domain=-4.3:7.3] plot(\x,{(-5.186588325271517--0.7625757947984444*\x)/-1.7885020370240152});
\draw [line width=1.pt,domain=-4.3:7.3] plot(\x,{(-6.252138214070486-0.7200313033000598*\x)/-2.5027203534373044});
\begin{scriptsize}
\draw [fill=black] (0.72,3.24) circle (2.5pt);
\draw [color=black,fill=white] (-1.94,1.94) circle (2.5pt);
\draw [fill=black] (0.4,2.06) circle (2.5pt);
\draw [fill=black] (1.80144262295082,2.1318688524590166) circle (2.5pt);
\draw [fill=black] (0.012940585926804804,2.894444647257461) circle (2.5pt);
\draw [color=black, fill=white] (0.5627203534373044,2.6600313033000598) circle (2.5pt);
\end{scriptsize}
\end{tikzpicture}
\caption{A supersolvable line arrangement with 2 modular points (shown as white dots).}
\label{FigSS}
\end{figure}

We refer to the number of lines of an arrangement containing a point as the multiplicity of the point.
So crossing points always have multiplicity at least 2. The modular points in Figure \ref{FigSS}
have multiplicity 3, while the other crossing points in the figure have multiplicity 2.
For $k\geq2$, we will use $t_k$ to denote the number of crossing points of multiplicity $k$.

For example, a pencil of $s$ lines has a unique modular point (indeed, a unique crossing point)
and it has multiplicity $s$, so $t_s=1$ and otherwise $t_k=0$. A near pencil of $s$ lines has $s$ modular points;
when $s > 3$, $s-1$ of the $s$ modular points have multiplicity 2 (so
$t_2=s-1$) and one has multiplicity $s-1$ (so $t_{s-1}=1$), while
if $s=3$ all the three modular points have multiplicity 2 (so $t_2 =
3$).  We will refer to
pencils and near pencils as trivial.

It is an open problem to determine which $t$ vectors $(t_2,\dots,t_s)$ can arise for real or complex
line arrangements, even for supersolvable line arrangements. 
It is even an open problem to classify all complex line arrangements with $t_2=0$, even for
the supersolvable case.
It is known that no nontrivial real line arrangement can have $t_2=0$. 
Three nontrivial kinds of complex line 
arrangements are known with $t_2=0$ but there is no proof that there are no others.
No nontrivial supersolvable complex
line arrangements are known with $t_2=0$, but again no proof is known that there are none
even though it is expected that $t_2$ is large. (In \cite{AT} it is conjectured that a nontrivial complex
supersolvable arrangement of $s$ lines has $t_2\geq s/2$.)

We will address these questions for real and for complex supersolvable line arrangements.
Our main results are Theorem \ref{BndsThm}, which shows that
a nontrivial complex line arrangement cannot have more than 4 modular points, and 
Theorem \ref{m=3,4}, which shows that
if all of the crossing points of a complex line arrangement have multiplicity 3 or 4, then 
the arrangement must have 0 modular points
(i.e., it cannot be supersolvable). This provides at least a little evidence for our 
Conjecture \ref{conj1} that every nontrivial complex supersolvable line arrangement
has at least one crossing point of multiplicity 2, and also supports Conjecture
\ref{conj2} \cite{AT} that there are in fact at least $s/2$ crossing points of multiplicity 2,
where $s$ is the number of lines in the arrangement.

The structure of the paper is as follows. In Section \ref{prelims} we recall
facts we will use later. In Section \ref{classification} we study the classification of
supersolvable real and complex line arrangements, and prove Theorem \ref{BndsThm}.
In Section \ref{Conjs} we consider various conjectures related to the occurrence
of points of multiplicity 2 on real and complex line arrangements
(such as Conjectures \ref{conj1} and \ref{conj2}), and we prove
Theorem  \ref{m=3,4}. Finally, in Section \ref{appls}, we discuss the application
of supersolvable line arrangements to the occurrence of unexpected plane curves,
and raise the question of whether all which can occur are already known.

\section{Preliminaries}\label{prelims}
Let ${\mathcal L}=\{L_1,\dots, L_s\}$ be a line arrangement in the
projective plane over an arbitrary field $K$. In this section we include some
well-known results that we use in this paper. 

Recall that, for every $k \ge 2$, $t_k$ denotes the number of crossing
points of multiplicity $k$. Let $n$ denote the number of all crossing
points. Let $m$ be the largest integer $k$ such that $t_k > 0$. 

First we have the following combinatorial identity which holds for any
field $K$. 
\begin{equation}\label{comb-id}
\binom{s}{2} = \sum_{k\geq2}\binom{k}{2}t_k.
\end{equation}

If $K = \mathbb{C}$ and $\mathcal{L}$ is nontrivial, we have the following
inequality due to Hirzebruch \cite{Hir83}. 

\begin{equation}\label{hirzebruch}
t_2+\frac{3}{4}t_3\geq s + \sum_{k>4}t_k(k-4).
\end{equation}

If $K = \mathbb{R}$ and $\mathcal{L}$ is not a pencil, we have the following inequality due to Melchior \cite{M}.

\begin{equation}\label{melchior}
t_2 \geq 3 +\sum_{k\geq3}(k-3)t_k.
\end{equation}

When $\operatorname{char}(K)=0$ and $\mathcal{L}$ is supersolvable, 
we have the following inequality proved in
\cite[Proposition 3.1]{AT}.
\begin{equation}\label{at}
t_2\geq 2n-m(s-m)-2. 
\end{equation}

The following result is \cite[Lemma 2.1]{T}. For the reader's convenience we include a proof.

\begin{lemma}\label{TohLem}
Let $\mathcal{L}$ be a supersolvable line arrangement (over any field $K$) with a modular point $p$ of multiplicity $m$.
If $q$ is a crossing point of multiplicity $n\geq m$, then $q$ is also modular.
\end{lemma}

\begin{proof}
In addition to the line $L=L_{p1}=L_{q1}$ through $p$ and $q$, $\mathcal{L}$ contains $m-1$ lines through $p$ 
(denote them by $L_{p2},\dots,L_{pm}$) and $n-1$ lines through $q$ (denote them by $L_{q2},\dots,L_{qn}$).
Let $r_{ij}$ be the point where $L_{pi}$ intersects $L_{qj}$.
Suppose $A$ and $B$ are any two distinct lines in $\mathcal{L}$.
Let $r$ be the point where $A$ and $B$ meet. We must show $r$ is on a line in $\mathcal{L}$ through $q$.
If either $A$ or $B$ contain $q$, then $r$ is on a line in $\mathcal{L}$ through $q$,
so assume neither $A$ nor $B$ contains $q$.

First say $n>m$. Let $a_j$ be the point where $A$ and $L_{qj}$ meet. Since $q\neq a_j$, we get
$n-1$ distinct points $a_j$, each of which is on some line $L_{pi_j}$ since $p$ is modular.
But there are only $m-1<n-1$ lines $L_{pi}$, so we must have $i_{j}=i_{j'}$ for some $j\neq j'$,
and hence $A=L_{pi_j}=L_{pi_{j'}}$, so $p\in A$. Likewise $p\in B$, so $r=p\in L$ 
is on a line in $\mathcal{L}$ through $q$. (This also shows that $n>m$ implies that
every line in $\mathcal{L}$ contains either $p$ or $q$; i.e., the lines in $\mathcal{L}$
are the $m+n-1$ lines through $p$ and $q$.)

Now say $n=m$. If both $A$ and $B$ contain $p$, then $r=p\in L$ is on a line in $\mathcal{L}$ through $q$.
So assume either $A$ or $B$ does not contain $p$; say $p\not\in A$.
But $p$ is modular, so the point $r$ where $A$ and $B$ meet is on $L_{pi'}$ for some $i'$.
Again, let $a_j$ be the point where $A$ and $L_{qj}$ meet. Since $q\neq a_j$, we get
$n-1$ distinct points $a_j$, each of which is on some line $L_{pi_j}$ since $p$ is modular.
If $i_j=i_{j'}$ for some $j\neq j'$, then $A=L_{pi_j}=L_{pi_{j'}}$, so $p\in A$ contrary to assumption.
Hence $i_j\neq i_{j'}$ whenever $j\neq j'$, the $n-1=m-1$ values of $j>1$ map under $j\mapsto i_j$ to
all $m-1=n-1$ values of $i>1$, hence for some $j'$ we have $i'=i_{j'}$, so $A$ meets $L_{pi'}$ at
$a_{j'}=r_{i_{j'}j'}=r_{i'j'}\in L_{pi'}$. But $A$ meets $L_{pi'}$ at $r\in L_{pi'}$, so $r=a_{j'}\in L_{qj'}$,
so $r$ is on a line in $\mathcal{L}$ through $q$. Thus $q$ is modular.
\end{proof}

\section{Classifying supersolvable line arrangements}\label{classification}

\subsection{Supersolvable line arrangements with modular points of multiplicity 2}
We first classify all line arrangements, over any field $K$, having one or more modular points of multiplicity 2,
or two (or more) modular points, not all of the same multiplicity.
Thus, after this section, we may assume all modular points have the same multiplicity, which is at least 3.

As a corollary of the proof of Lemma \ref{TohLem}, we have the following result, 
which classifies line arrangements where at least two multiplicities occur as multiplicities of modular points.

\begin{corollary}\label{n>m Cor}
Let $\mathcal{L}$ be a supersolvable line arrangement (over any field $K$) with a modular point $p$ of multiplicity $m$.
If $q$ is a crossing point of multiplicity $n>m$, then $\mathcal{L}$ consists exactly of the $m$ lines through $p$
and the $n$ lines through $q$ (hence $m+n-1$ lines altogether). 
If $m=2$, then every crossing point is modular and  $\mathcal{L}$ is a near pencil.
If $m>2$, then the only modular points are $p$ and $q$.
\end{corollary}

\begin{proof}
We saw in the proof of Lemma \ref{TohLem} that the lines in $\mathcal{L}$
are the $m+n-1$ lines through $p$ and $q$. If $m=2$, the only lines are the $m$ lines through
$q$ (one of which goes through $p$) and the remaining line through $p$, hence 
$\mathcal{L}$ is a near pencil. If $m>2$, there are $(n-1)(m-1)$ crossing points 
of multiplicity 2, but a point of multiplicity 2 on one line through $q$ is connected to at most
one point of multiplicity 2 on any other line through $q$, and hence no point of
multiplicity 2 is modular. I.e., the only modular crossing points are $p$ and $q$.
\end{proof}

\begin{proposition}
Let $\mathcal L$ be a line arrangement (over any field) having one or more modular points, 
exactly one of which has multiplicity 2 (call this point $p$). Then $\mathcal L$ is the pencil consisting of the
two lines through $p$.
\end{proposition}

\begin{proof} 
If $\mathcal L$ had a crossing point of multiplicity $n>m=2$, then by Corollary \ref{n>m Cor},
$\mathcal L$ is a near pencil, and thus would have $n$ points of multiplicity 2.
Thus $\mathcal L$ has exactly one crossing point, and it has multiplicity 2,
so $\mathcal L$ is the pencil consisting of the
two lines through $p$.
\end{proof}

\begin{proposition}
Let $\mathcal L$ be a line arrangement (over any field) having two or more modular points, 
at least two of which have multiplicity 2. Then $\mathcal L$ is a near pencil.
\end{proposition}

\begin{proof} Let $p$ and $q$ be modular points of multiplicity 2.
Since $\mathcal L$ is supersolvable, given a crossing point 
other than $p$, the line from $p$ to that point
is in $\mathcal L$. But $p$ has multiplicity 2, so every crossing point 
must be on one or the other of the two lines through $p$.
Likewise, every crossing point 
must be on one or the other of the two lines through $q$.

Let $L$ be the line through both $p$ and $q$; thus $L\in \mathcal L$.
Let $L_p$ be the other line in $\mathcal L$ through $p$ and 
let $L_q$ be the other line in $\mathcal L$ through $q$.
Let $r$ be the point where $L_p$ and $L_q$ meet.
Thus any crossing point not on $L$ must be on both 
$L_p$ and $L_q$; i.e., $r$ is the only crossing point not on $L$.
Thus every line in $\mathcal L$ other than $L$ must contain $r$,
so $\mathcal L$ is a near pencil.
\end{proof}

\subsection{Homogeneous supersolvable line arrangements (mostly for $\operatorname{char}(K)=0$)}
By our foregoing results, we see that it remains to understand
supersolvable line arrangements such that all modular points have
the same multiplicity $m$ (we say such a supersolvable line arrangement is
{\em homogeneous} or {\em $m$-homogeneous}) with $m\geq 3$. It follows from Lemma \ref{TohLem} that
$t_k=0$ for $k>m$. For an $m$-homogeneous supersolvable line arrangement $\mathcal L$,
we denote $m$ by $m_{\mathcal L}$.

\subsubsection{The values of $t_{m_{\mathcal L}}$ that arise for $\operatorname{char}(K)=0$} 
When $K$ is algebraically closed but of finite characteristic, there is no bound
to the number of modular points a supersolvable line arrangement can have.
(Just take all lines defined over a finite field $F$ of $a$ elements.
Then the arrangement has $a^2+a+1$ lines and the same number of crossing points;
all are modular and all have multiplicity $a+1$.)
In characteristic 0 things are very different, as we show in Theorem \ref{BndsThm}.

To prove the theorem, we will use the following result.

\begin{proposition}\label{no3OnLine}
For an $m$-homogeneous supersolvable complex line arrangement $\mathcal L$ with $m=m_{\mathcal L}\geq3$,
no three modular points are collinear.
\end{proposition}

\begin{proof}
Suppose that $p, q$ and $r$ are collinear modular points. Then the line $L$ that contains them is in 
$\mathcal L$. Moreover, $\mathcal L$ contains $m-1$ additional lines through each of $p,q$ and $r$.
Denote the union of these $m-1$ lines through $p$ by $C_p$. Similarly, we have $C_q$ and $C_r$.
The intersection of the curves $C_p$ and $C_q$ is a complete intersection of $(m-1)^2$ points,
which are also contained in $C_r$. Since the curves all have degree $m-1$, we see that
$C_r$ is in the pencil defined by $C_p$ and $C_q$. I.e., the forms $F_p, F_q$ and $F_r$
defining the curves are such that $F_r$ is a linear combination of $C_p$ and $C_q$.
we can choose coordinates such that $L$ is $x=0$, $p$ is $x=y=0$, $q$
is $x=z=0$ and $r$ is $y=z=1$. In terms of these coordinates, the restrictions of
$F_p,F_q,F_r$ to $L$ are $y^{m-1}$, $z^{m-1}$ and $ay^{m-1}+bz^{m-1}=(y-z)^{m-1}$
for some nonzero constants $a$ and $b$. Setting $z=1$, we thus see that
$ay^{m-1}+b=(y-1)^{m-1}$, so $ay^{m-1}+b$ has a multiple root at $y=1$.
This contradicts the fact that the derivative $a(m-1)y^{m-2}+b$ is not 0 at $y=1$.
\end{proof}

\begin{theorem}\label{BndsThm}
For an $m$-homogeneous supersolvable complex line arrangement $\mathcal L$ with $m=m_{\mathcal L}\geq3$,
we have $1\leq t_m\leq 4$.
\end{theorem}

\begin{proof}
First we show that $t_m<7$. Suppose $t_m\geq 7$ for some $m\geq 3$. Each 
non-modular crossing point is connected by a line 
to each of the $t_m\geq 7$ modular points. Since at most two modular points can lie on any line
by Proposition \ref{no3OnLine}, we see that each crossing point must have multiplicity at least 4.
Also, each modular point has multiplicity $m\geq 6$ since each one connects to each of the others.
Thus $t_2=t_3=0$, but this is impossible by Inequality \eqref{hirzebruch}.

Next we show that $t_m<6$. Suppose $\mathcal L$ has $t_m=6$. 
It is enough to show $t_m<6$ under the assumption that every line in 
${\mathcal L}$ contains a modular point. (This is because if we let ${\mathcal L}'$ be
the line arrangement obtained from ${\mathcal L}$ by deleting all lines not through
a modular point, ${\mathcal L}'$ still has $t_{m_{{\mathcal L}'}}=6$.)
Since every modular point is on a line in ${\mathcal L}$ through another modular point, 
we have $m\geq 5$. Every crossing point $q$ of ${\mathcal L}$ also connects to every modular point
so has multiplicity at least 3 (since a line can go through at most 2 modular points), with multiplicity exactly
3 if and only if $q$ is 3 lines through pairs of modular points. 

There are $2\binom{6}{4}=30$ possible locations for crossing points of multiplicity 3, hence $t_3\leq 30$.
To see this note that there are $\binom{6}{4}$ ways to pick 4 of the 6 modular points. There are 
3 reducible conics through these 4 points. The singular points of these three conics
are crossing points where two lines through disjoint pairs chosen from the 4 points 
intersect. In order to get a point $q$ of multiplicity 3, the line $H$ through the remaining 2 points of the 6 modular points
must contain $q$. This might not happen for any of the three singular points,
but it can be simultaneously true for at most two of the three singular points, since at most
two of the singular points can be on the line $H$  (this is merely because the three singular points
cannot be collinear in characteristic 0). Thus we get at most $2\binom{6}{4}=30$ possible locations 
for crossing points of multiplicity 3. 

Now apply Inequality \eqref{hirzebruch}, using the fact that our assumption (that every line in 
${\mathcal L}$ contains a modular point) implies that ${\mathcal L}$ has $(m-5)6+\binom{6}{2}$ lines:
$$22.5=\frac{3}{4}30\geq \frac{3}{4}t_3\geq ((m-5)6+\binom{6}{2})+(m-4)6.$$
For $m\geq 6$ this is $22.5\geq 12m-39\geq 33$, thus the only possibility for $t_m=6$ is $m=5$.
For $m=5$ we see ${\mathcal L}$ has $\binom{6}{2}=15$ lines and every crossing point
has multiplicity at least 3 and at most 5, so from Equation \eqref{comb-id} we get:
$$105=\binom{15}{2}=3t_3+6t_4+10t_5=3t_3+6t_4+60$$
so $15=t_3+2t_4$, hence $t_3\leq 15$. Inequality \eqref{hirzebruch} now gives
$(3/4)15\geq 15+6$, which is false.

Finally, we show that $t_m<5$. So assume $t_m=5$. Arguing as before, we may assume that every line in 
${\mathcal L}$ contains a modular point. We still have that al non-modular points have multiplicity at least 3,
and the 5 modular points have multiplicity $m\geq4$. Each choice of 4 of the 5 modular points gives
3 possible locations for a triple point, hence $t_3\leq 3(5)=15$. Thus Inequality \eqref{hirzebruch} gives
$11.25=(3/4)15\geq (\binom{5}{2}+(m-4)5)+(m-4)5=10m-30$, which is impossible for $m\geq 5$.
For $m=4$ we see ${\mathcal L}$ has $\binom{5}{2}=10$ lines and every crossing point
has multiplicity at at least 3 and at most 4, so from Equation \eqref{comb-id} we get:
$$45=\binom{10}{2}=3t_3+6t_4=3t_3+30$$
so $5=t_3$. Inequality \eqref{hirzebruch} now gives
$(3/4)5\geq 10$, which is false.
\end{proof}

\begin{example}
For $m$-homogeneous supersolvable line arrangements
over both the complex numbers and the reals, all four cases $1\leq t_{m_{\mathcal L}}\leq 4$ arise. 
It is easy to obtain examples with
exactly one modular point; see Section \ref{singleModPt}. (However, the fact that there are many examples
makes it hard to classify them!) It is also easy to obtain examples with exactly two modular points;
see Corollary \ref{n>m Cor}. For exactly three modular points, consider the line arrangement 
defined by the linear factors of $xyz(x^n-y^n)(x^n-z^n)(y^n-z^n)$ for $n\geq2$.
The coordinate vertices are the modular points, and have multiplicity $n+2$. For $n=2$ the arrangement is real
(see the arrangement of 9 lines shown in Figure \ref{FigSS 3 3});
for $n>2$ it is complex but not real. Taking $n=1$, so $xyz(x-y)(x-z)(y-z)$, gives the only example
we know over the complexes or reals with exactly four modular points; see Case 2 of Figure \ref{Fig m=3}.
(We thank \c{S}.\ Toh\v{a}neanu for pointing out that a line arrangement equivalent to the one
defined by the linear factors of $xyz(x^n-y^n)(x^n-z^n)(y^n-z^n)$ for $n=2$ arose as 
an example in section 3.1.1 of \cite{AT}, to show that a certain bound on the number of
crossing points was sharp. For the line arrangements given by 
$xyz(x^n-y^n)(x^n-z^n)(y^n-z^n)$ the bound is $s\leq d^2+d+1$, where $s=n^2+3n+3$ is the number of 
crossing points and $d=m_{\mathcal L}-1=n+1$. Thus we see that $s=d^2+d+1$, so 
this bound is in fact sharp for all values of $n$.)
\end{example}

\subsubsection{Classifying $m$-homogeneous $\mathcal L$ for $t_m>1$ and $m=3$}
Consider the case of a line arrangement $\mathcal L$ with two or more modular points
of multiplicity $m\geq 3$. Since we have at least two modular points,
we pick two and call them $p$ and $q$.

First say $m=3$. We will show that there are three cases, shown in Figure \ref{Fig m=3}:
$\mathcal L$ has either 5, 6 or 7 lines, and either 2, 4 or 7 modular points, respectively.
The case of 7 lines occurs only in characteristic 2. The other cases occur for any field.

Clearly, $\mathcal L$ has at least 5 lines:
the line $L$ defined by $p$ and $q$,
and in addition lines $p\in L_{pi}$ and $q\in L_{qi}$, for $i=1,2$.
No other lines in $\mathcal L$ (if any) can contain $p$ or $q$.
Let $r_1$ be where $L_{p1}$ and $L_{q1}$ meet and let
$r_2$ be where $L_{p2}$ and $L_{q2}$ meet.
And let $s_1$ be where $L_{p1}$ and $L_{q2}$ meet and let
$s_2$ be where $L_{p2}$ and $L_{q1}$ meet.
Any other line in $\mathcal L$ must intersect the lines
$L_{pi}$ and $L_{qi}$ only at
$r_1,r_2,s_1,$ or $s_2$. 

One possibility is that $\mathcal L$ has only the five lines mentioned above.
Alternatively, assume $\mathcal L$ has another line, $A$. Of the six pairs two points chosen
from the four points $r_1,r_2,s_1$ and $s_2$, $A$ must contain
either $r_1$ and $r_2$ or $s_1$ and $s_2$ ($A$ cannot contain $r_1$ and $s_1$,
for example, because that line is $L_{p1}$).
Up to relabeling, the case $r_1$ and $r_2$ is the same as $s_1$ and $s_2$,
so say $A$ contains $r_1$ and $r_2$. 
Up to projective equivalence, we may assume that
$p=(0,0,1)$, $q=(0,1,0)$, $r_1=(1,0,0)$ and
$r_2=(1,1,1)$, in which case $s_1=(1,0,1)$ and $s_2=(1,1,0)$. So a second possibility is that 
$\mathcal L$ has six lines, with $A$ being the sixth line.
Note that in this case that $\mathcal L$ has 4 modular points:
the points $p,q,r_1$ and $r_2$ are modular, and all have multiplicity 3.
The only option for $\mathcal L$ to contain an additional line is for 
the additional line (call it $B$) to be the line through $s_1$ and $s_2$.
But $A$ is $y-z=0$ and $B$ is $x-y-z=0$, so 
$A$ and $B$ intersect at the point $(2,1,1)$.
When the ground field does not have characteristic 2, this is not on
any of the three lines through $p$ (or on any of the three lines through $q$),
hence including $B$ would make $\mathcal L$ not be supersolvable.
Thus when the characteristic is not 2, $\mathcal L$ must either have 5 or 6 lines,
and be Case 1 or Case 2 shown in Figure \ref{Fig m=3}.
If the characteristic is 2, the point $(2,1,1)$ is on the line through $p$ and $q$,
in which case $\mathcal L$ consists of the 7 lines of the Fano plane,
there are 7 crossing points, all are modular and have multiplicity 3.

\begin{figure}
\hbox to\hsize{\hfil\hbox{\begin{tikzpicture}[line cap=round,line join=round,x=0.5cm,y=0.5cm]
\clip(-3,-1) rectangle (3,5.5);
\draw [line width=1.pt,domain=-4.3:7.3] plot(\x,{(--8.928203230275512--3.4641016151377553*\x)/2.});
\draw [line width=1.pt,domain=-4.3:7.3] plot(\x,{(--8.928203230275509-3.4641016151377553*\x)/2.});
\draw [line width=1.pt,domain=-4.3:7.3] plot(\x,{(-4.-0.*\x)/-4.});
\draw [line width=1.pt,domain=-4.3:7.3] plot(\x,{(-7.464101615137757--2.*\x)/-3.4641016151377553});
\draw [line width=1.pt,domain=-4.3:7.3] plot(\x,{(-7.464101615137757-2.*\x)/-3.4641016151377553});
\begin{scriptsize}
\draw [fill=white] (-2.,1.) circle (2.5pt);
\draw [fill=white] (2.,1.) circle (2.5pt);
\draw [fill=black] (0.,4.464101615137755) circle (2.5pt);
\draw [fill=black] (-1.,2.7320508075688776) circle (2.5pt);
\draw [fill=black] (1.,2.7320508075688776) circle (2.5pt);
\draw [fill=black] (0.,2.154700538379252) circle (2.5pt);
\draw[color=black] (0,0) node {Case 1};
\end{scriptsize}
\end{tikzpicture}}\hfil
\hbox{\begin{tikzpicture}[line cap=round,line join=round,x=0.5cm,y=0.5cm]
\clip(-3,-1) rectangle (3,5.5);
\draw [line width=1.pt,domain=-4.3:7.3] plot(\x,{(--8.928203230275512--3.4641016151377553*\x)/2.});
\draw [line width=1.pt,domain=-4.3:7.3] plot(\x,{(--8.928203230275509-3.4641016151377553*\x)/2.});
\draw [line width=1.pt,domain=-4.3:7.3] plot(\x,{(-4.-0.*\x)/-4.});
\draw [line width=1.pt] (0.,0.5) -- (0.,6.3);
\draw [line width=1.pt,domain=-4.3:7.3] plot(\x,{(-7.464101615137757--2.*\x)/-3.4641016151377553});
\draw [line width=1.pt,domain=-4.3:7.3] plot(\x,{(-7.464101615137757-2.*\x)/-3.4641016151377553});
\begin{scriptsize}
\draw [fill=white] (-2.,1.) circle (2.5pt);
\draw [fill=white] (2.,1.) circle (2.5pt);
\draw [fill=white] (0.,4.464101615137755) circle (2.5pt);
\draw [fill=black] (-1.,2.7320508075688776) circle (2.5pt);
\draw [fill=black] (1.,2.7320508075688776) circle (2.5pt);
\draw [fill=white] (0.,2.154700538379252) circle (2.5pt);
\draw [fill=black] (0.,1) circle (2.5pt);
\draw[color=black] (0,0) node {Case 2};
\end{scriptsize}
\end{tikzpicture}}\hfil
\hbox{\begin{tikzpicture}[line cap=round,line join=round,x=0.5cm,y=0.5cm]
\clip(-3,-1) rectangle (3,5.5);
\draw [line width=1.pt] (0.,2.154700538379252) circle (0.577cm);
\draw [line width=1.pt,domain=-4.3:7.3] plot(\x,{(--8.928203230275512--3.4641016151377553*\x)/2.});
\draw [line width=1.pt,domain=-4.3:7.3] plot(\x,{(--8.928203230275509-3.4641016151377553*\x)/2.});
\draw [line width=1.pt,domain=-4.3:7.3] plot(\x,{(-4.-0.*\x)/-4.});
\draw [line width=3.pt,color=white] (0.,.5) -- (0.,6.3);
\draw [line width=1.pt] (0.,0.5) -- (0.,6.3);
\draw [line width=3.pt,domain=-4.3:7.3,color=white] plot(\x,{(-7.464101615137757--2.*\x)/-3.4641016151377553});
\draw [line width=1.pt,domain=-4.3:7.3] plot(\x,{(-7.464101615137757--2.*\x)/-3.4641016151377553});
\draw [line width=3.pt,domain=-4.3:7.3,color=white] plot(\x,{(-7.464101615137757-2.*\x)/-3.4641016151377553});
\draw [line width=1.pt,domain=-4.3:7.3] plot(\x,{(-7.464101615137757-2.*\x)/-3.4641016151377553});
\begin{scriptsize}
\draw [fill=white] (-2.,1.) circle (2.5pt);
\draw [fill=white] (2.,1.) circle (2.5pt);
\draw [fill=white] (0.,4.464101615137755) circle (2.5pt);
\draw [fill=white] (-1.,2.7320508075688776) circle (2.5pt);
\draw [fill=white] (1.,2.7320508075688776) circle (2.5pt);
\draw [fill=white] (0.,2.154700538379252) circle (2.5pt);
\draw [fill=white] (0.,1) circle (2.5pt);
\draw[color=black] (0,0) node {Case 3 (char 2)};
\end{scriptsize}
\end{tikzpicture}}\hfil}
\caption{Classification of supersolvable line arrangements with 2 or more modular points
(shown as white dots), all of multiplicity $m=3$.}
\label{Fig m=3}
\end{figure}
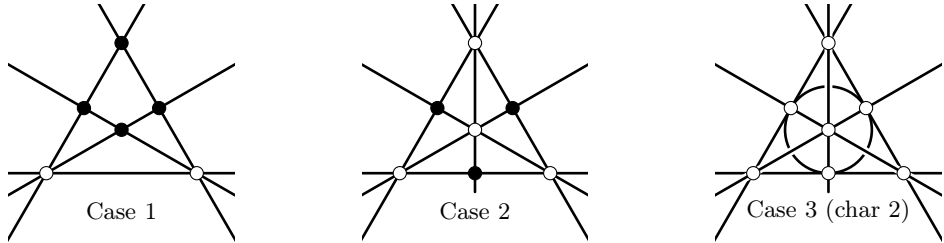

\subsubsection{Classifying $m$-homogeneous $\mathcal L$ over the reals for $t_m>1$ and $m>3$}
Now we consider the case $m\geq 4$ for real line arrangements.
So, in addition to the line $L$ through $p$ and $q$, there are $m-1$ lines through
$p$ and $m-1$ lines through $q$. These lines form a complete intersection (i.e., a grid)
of $(m-1)^2$ crossing points. The only other crossing points for these
$2m-1$ lines are $p$ and $q$. Certainly $\mathcal L$ could consist of only these
$2m-1$ lines, in which case $p$ and $q$ are the only modular points
and we have $t_k=0$ except for $t_m=2$ and $t_2=(m-1)^2$.

The question now is what additional lines can be added to these $2m-1$
while maintaining supersolvability. To answer this, let's choose coordinates so that
$p$ becomes $(0,1,0)$ and $q$ becomes $(1,0,0)$. Thus the line through $p$ and $q$
is now the line at infinity, and the $m-1$ other lines through $p$ are parallel to the $x=0$ axis,
and the $m-1$ other lines through $q$ are parallel to the $y=0$ axis.

Any additional line must avoid $p$ and $q$, and must intersect the $m-1$ vertical
lines only at points where they meet the $m-1$ horizontal lines. By inspection we can see that this 
can happen in exactly to ways. First is that the four corners of the grid form a rectangle and the $i$th vertical line
(counting from the left) meets the $i$th horizontal line (counting up from the bottom)
meet on the anti-diagonal of the rectangle (in which case the anti-diagonal can be added to
$\mathcal L$). The second way is that the four corners of the grid form a rectangle (as before) and the $i$th vertical line
(counting from the left) meets the $i$th horizontal line (counting DOWN this time from the top)
meet on the main diagonal of the rectangle (in which case the main diagonal can be added to
$\mathcal L$). In case both cases hold, both diagonals can be added if and only if $m$ is even.

Thus there are three cases: $\mathcal L$ has $2m-1$ lines and we have $t_m=2$ and $t_2=(m-1)^2$
but only two modular points, namely $p$ and $q$;
$\mathcal L$ has $2m$ lines where the additional line is one of the two major diagonals
(assuming the lines are spaced correctly)
and we have still have only two modular points ($p$ and $q$), with $t_m=2$, $t_2=(m-1)^2-(m-1)+1$; or
$\mathcal L$ has $2m+1$ lines where the additional lines are the two major diagonals
(assuming the lines are spaced correctly and $m$ is even), in which case either $m=4$ and we have
$t_m=3$, $t_2=6$, $t_3=4$ and there are three modular points ($p$, $q$ and the center of the rectangle),
or $m>4$ and we have
$t_m=2$, $t_2=(m-1)^2-(2m-1)+2$, $t_3=2m-4$ and $t_4=1$ 
and there are only two modular points ($p$ and $q$).

Thus we have a complete classification of real supersolvable line arrangements when there is more than one
modular point of multiplicity at least 3. 

\subsubsection{Classifying $m$-homogeneous $\mathcal L$ over the complexes for $t_m>2$ and $m>3$}
Now we consider the case $m\geq 4$ for complex line arrangements with
at least 3 modular points. By Theorem \ref{BndsThm}, 
the number of modular points cannot be more 
than 4.

We begin with the case of exactly $t_m=3$ modular points.
If $\mathcal L$ has a line that does not contain a modular point, deleting it
gives an arrangement which is still supersolvable, so we first assume
every line in $\mathcal L$ goes through a modular point.

After a change of coordinates, we may assume that the three modular points, $p,q,r$,
are the coordinate vertices of ${\mathbb P}^2$, so say $p=(0,0,1), q=(0,1,0), r=(1,0,0)$.
In addition to the three coordinate axes, $\mathcal L$ must contain $m-2$ lines through
each of $p$, $q$ and $r$. Let $F_p$ be the form defining the union of
these $m-2$ lines through $p$, other than the coordinate axes. Note
that $F_p$ is a form of degree $m-2$ and
involves only the variables $x$ and $y$, hence is $F_p(x,y)$. Likewise we have
$F_q(x,z)$ and $F_r(y,z)$ for $q$ and $r$. Since the coordinate axes are not among the lines
defined by $F_p$, $F_q$ or $F_r$, we see that none of these forms is divisible by a variable.

The crossing points for the lines from
$F_p$ and the lines from $F_q$ form a complete intersection of $(m-2)^2$ points
on which $F_r$ also vanishes, so $F_r=aF_p+bF_q$ for some scalars $a$ and $b$.
The only term that $F_p$ and $F_q$ can have in common is $x^{m-2}$. Thus in order that
all terms involving $x$ cancel in $aF_p+bF_q$ so that $F_r$ involves only $y$ and $z$, 
we see that $x^{m-2}$ is the only term in either $F_p$ or $F_q$ involving $x$.
Thus (after dividing by the coefficient of $x^{m-2}$ in each case) 
we have $F_p=x^{m-2}-\alpha y^{m-2}$ and $F_q=x^{m-2}-\beta z^{m-2}$.
By absorbing the $\alpha$ into $y$ and the $\beta$ into $z$, we get
$F_p=x^{m-2}-y^{m-2}$ and $F_q=x^{m-2}-z^{m-2}$, so
$F_r=y^{m-2}-z^{m-2}$. 

Thus if every line in $\mathcal L$ goes through one of the three modular points,
then the lines in $\mathcal L$ correspond to the linear factors of
$xyz(x^{m-2}-y^{m-2})(x^{m-2}-z^{m-2})(y^{m-2}-z^{m-2})$.
Now we check that no line not through $p$, $q$ or $r$ can be added to $\mathcal L$ 
while still preserving
supersolvability. If such a line $L$ existed, it would need to intersect every line
of $\mathcal L$ in a crossing point. In particular, $L$ must contain one of 
the $(m-2)^2$ intersection points of the lines from $F_p$ and the lines from $F_q$. 
Let $n: = m-2$. By an appropriate change of coordinates
obtained by multiplying $x,y$ and $z$ by appropriate powers of 
an $n$th root of 1, we may assume that $L$ contains $(1,1,1)$.
Let $\epsilon=\cos(2\pi/n)+\imath \sin(2\pi/n)$
be a primitive $n$th root of 1.
The line $L$ must intersect $y-\epsilon z=0$ at a crossing point 
(hence at $(\epsilon^i,\epsilon,1)$ for some $1 \le i \le n$)
and also $y-\epsilon^2 z=0$ at a crossing point
(hence at $(\epsilon^j,\epsilon^2,1)$ for some $1 \le j \le n$).
The question is whether $i$ and $j$ exist such that these points
lie on a line through $(1,1,1)$ which does not go through $p$, $q$ or $r$.

The lines through $(1,1,1)$ are of the form $a(x-z)+b(y-z)=0$. For the line not
to go through $p$, $q$ or $r$, we need $ab\neq0$. Thus we can write the line
as $c=(y-z)/(x-z)$ for some $c\neq 0$.
For $(\epsilon^i,\epsilon,1)$ and $(\epsilon^j,\epsilon^2,1)$ both to lie on this line
we must have 
$$\frac{\epsilon-1}{\epsilon^i-1}=\frac{\epsilon^2-1}{\epsilon^j-1}.$$
This simplifies to 
$$\epsilon^{i-1}(\epsilon+1)=\epsilon^{j-1}+1.$$
Thus the complex norms are equal; i.e.,  $|\epsilon+1|=|\epsilon^{j-1}+1|$.
But if $\gamma=\cos(\theta)+\imath \sin(\theta)$, the norm $|\gamma+1|$ is a decreasing function
of $\theta$ for $0\leq\theta\leq\pi$, so the only possibilities
for $|\epsilon+1|=|\epsilon^{j-1}+1|$ are $j=2, n$. 
If $j=2$, then $\epsilon^{i-1}(\epsilon+1)=\epsilon^{j-1}+1$
forces $i=1$, so the line through $(\epsilon^i,\epsilon,1)$ and $(\epsilon^j,\epsilon^2,1)$
then is $x-y=0$, which contains $p$. If $j=n$, then 
$\epsilon^{i-1}(\epsilon+1)=\epsilon^{j-1}+1=(1+\epsilon)/\epsilon$ forces $\epsilon^i=1$.
and hence $i=n$, so the line is $x-z=0$, which contains $q$.

Thus the only possibility for 3 modular points of multiplicity $m>3$, is (up to choice of coordinates)
for the line arrangement to be the lines defined by the linear factors of
$xyz(x^{m-2}-y^{m-2})(x^{m-2}-z^{m-2})(y^{m-2}-z^{m-2})$.

Now suppose $\mathcal L$ has 4 modular points with $m>3$. We can, up to choice
of coordinates, assume that the four points are $p,q,r,s$, where $p,q,r$ are as above, and $s=(1,1,1)$.
If we delete any lines not through $p,q,r$, then the resulting arrangement must
come from the linear factors of $xyz(x^{m-2}-y^{m-2})(x^{m-2}-z^{m-2})(y^{m-2}-z^{m-2})$. 
To get $\mathcal L$, we must add back in lines through $s$ which intersect the lines coming from
$xyz(x^{m-2}-y^{m-2})(x^{m-2}-z^{m-2})(y^{m-2}-z^{m-2})$ only at crossing points
for the lines from $xyz(x^{m-2}-y^{m-2})(x^{m-2}-z^{m-2})(y^{m-2}-z^{m-2})$.
But as we just saw there are no such lines. Thus $\mathcal L$ having 4 modular points with $m>3$
is impossible.

Thus, up to choice of coordinates, the only complex supersolvable line arrangement with 4 modular points
is the one we found before; i.e., $xyz(x^{m-2}-y^{m-2})(x^{m-2}-z^{m-2})(y^{m-2}-z^{m-2})$ with $m=3$,
displayed in Case 2 of Figure \ref{Fig m=3}. 
And up to choice of coordinates the only complex supersolvable line arrangements with 3 modular points
are given by the linear factors of $xyz$ when $m=2$, and by the linear factors of 
$xyz(x^{m-2}-y^{m-2})(x^{m-2}-z^{m-2})(y^{m-2}-z^{m-2})$
for $m > 3$.

We do not have a classification of complex supersolvable line arrangement with just 1 or 2 modular points.
If for $m\geq 3$ you remove one or more of the linear factors of $y^{m-2}-z^{m-2}$ from the set
of linear factors of $xyz(x^{m-2}-y^{m-2})(x^{m-2}-z^{m-2})(y^{m-2}-z^{m-2})$, then we get examples
of complex supersolvable line arrangement with just 2 modular points. Thus more examples occur 
over the complexes than over the reals, but it is not clear what the full range of possibilities is.

In any case, we have given a full classification over the complexes for 
supersolvable line arrangement with 3 or 4 modular points. We discuss the case of 1 modular point
in the next section.

\subsection{Having a single modular point}\label{singleModPt}
The case that there is a single modular point is the hardest to classify and we can give only partial
results in this case.

We begin with a lemma.

\begin{lemma}\label{maxMultLem}
Let $\mathcal L$ be a line arrangement (not necessarily supersolvable, not necessarily over the reals).
Let $m$ be the maximum of the multiplicities of the crossing points and let $n$ be the number of 
crossing points. If $n<2m$, then $\mathcal L$ is either a pencil or near pencil.
\end{lemma}

\begin{proof}
Assume $\mathcal L$ is not a pencil or a near pencil. Let $p$ be a point of multiplicity $m$
and take lines $A$ and $B$ not through $p$. Then $A$ and the $m$ lines through $p$ give
$m+1$ crossing points, and $B$ then gives at least another $m-1$ crossing points,
for a total of at least $2m$ crossing points.
\end{proof}

We now consider the case of a line arrangement $\mathcal L$ with a single modular point,
which we assume has multiplicity $m>2$; call it $p$. By \cite{AT}
every other crossing point of $\mathcal L$ has multiplicity less than $p$
(because for a supersolvable line arrangement, all points of maximum multiplicity 
are modular). Assume $\mathcal L$ is not a pencil or a near pencil.
Let $\mathcal L'$ be the arrangement obtained from $\mathcal L$
by removing the $m$ lines through $p$. We can recover $\mathcal L$
by adding to $\mathcal L'$ every line from $p$ to a crossing point of $\mathcal L'$.
What is difficult to know is how many lines get added, since one line through $p$
might contain more than one crossing point of $\mathcal L'$.
But we see that $t_m=1$ and $t_{k+1}=t_k'$ for all $2<k<m$, where $t_k'$
is the number of crossing points of $\mathcal L'$ of multiplicity $k$.
Even knowing how many lines are in $\mathcal L'$ and the value of $t_k'$ for every $k$, it's hard to
say how many lines are in $\mathcal L$, or what the value of $t_2$ is,
except in certain special situations. 

Suppose, for example, we know that no two crossing points of $\mathcal L'$ are on the same line through
$p$. Since $\mathcal L'$ has $t_2'+\dots+t_m'$ crossing points and $\mathcal L'$ 
has $s'$ lines, where $\binom{s'}{2}=\sum_kt_k'\binom{k}{2}$ (see \eqref{comb-id}), we then know that
$\mathcal L$ has $s=s'+t_2'+\dots+t_m'$ lines and then from 
$\binom{s}{2}=\sum_kt_k\binom{k}{2}$ we can determine $t_2$.

Alternatively, start with any line arrangement $\mathcal L'$ (over any field) which is not a pencil or a near pencil.
By Lemma \ref{maxMultLem}, $n'\geq2m'$, where $n'$ is the number of crossing points of 
$\mathcal L'$ and $m'$ is the maximum of their multiplicities.
For a general point $p$, no line through $p$ will contain more than one crossing point of 
$\mathcal L'$. Now add to $\mathcal L'$ each line from $p$ to a crossing point of $\mathcal L'$
to get a larger line arrangement $\mathcal L$ of $s=n'+s'$ lines, where $s'$ is the number of lines of $\mathcal L'$.
We also know that $t_{k+1}=t_k'$ for all $k>2$, and we can determine $t_2$ from
$\binom{s}{2}=\sum_kt_k\binom{k}{2}$. Moreover, $p$ is the unique modular point of $\mathcal L$.
Note that $p$ has multiplicity $n'\geq 2m'$ and the maximum multiplicity of any other crossing point of
$\mathcal L$ is $m'+1<2m'$. Thus if $\mathcal L$ has another modular point, it has multiplicity
$d<n'$, hence by our classification $\mathcal L$ has $d+n'-1$ lines. But in fact $s'\geq d+1$
since $\mathcal L'$ is not a pencil or near pencil, and
$\mathcal L$ has $s=s'+n'>d+1-n'$ lines. Thus $\mathcal L$ has a unique modular point, namely $p$.
Thus classifying line arrangements with a unique modular point, even when that point is general,
comes down to classifying line arrangements in general.

\subsection{Summary}

The real supersolvable line arrangements having more than one modular point
can be subsumed by one general construction. Take two points, $p$ and $q$, on a line $L$.
Take $a_p\geq 0$ additional lines through $p$ and $a_q\geq0$ additional lines through $q$.
This gives a supersolvable line arrangement as long as $a_p+a_q>0$.
In addition, if $a_p=a_q\geq 2$ and the obvious collinearity condition obtains,
an additional line can be added in two possible ways (shown by the dashed and dashed-dotted lines 
in Figure \ref{FigSS 3 3} in the case of $a_p=a_q=3$). If both can be added separately and if 
$a_p=a_q$ is odd, both can be added simultaneously.
These constructions cover all possible cases of real supersolvable line arrangements
with 2 or more modular points.

The case of complex supersolvable line arrangements with more than two modular points 
are all given, up to choice of coordinates, by the linear factors of 
$xyz(x^{m-2}-y^{m-2})(x^{m-2}-z^{m-2})(y^{m-2}-z^{m-2})$ for $m\geq3$.

\begin{figure}
\definecolor{uuuuuu}{rgb}{0.26666666666666666,0.26666666666666666,0.26666666666666666}
\definecolor{ududff}{rgb}{0.30196078431372547,0.30196078431372547,1.}
\begin{tikzpicture}[line cap=round,line join=round,x=1.5cm,y=1.5cm]
\clip(-3.3328037910523833,0) rectangle (4.790630352259127,4);
\draw [line width=1.pt,domain=-3.3328037910523833:4.790630352259127] plot(\x,{(--9.2088--3.42*\x)/3.1});
\draw [line width=1.pt,domain=-3.3328037910523833:4.790630352259127] plot(\x,{(--9.1608--2.98*\x)/4.28});
\draw [line width=1.pt,domain=-3.3328037910523833:4.790630352259127] plot(\x,{(-4.840644205324093--3.502567576181975*\x)/-1.5949665690340673});
\draw [line width=1.pt,domain=-3.3328037910523833:4.790630352259127] plot(\x,{(-4.5825417994860755--2.4425675761819754*\x)/-2.1549665690340674});
\draw [line width=1.pt,domain=-3.3328037910523833:4.790630352259127] plot(\x,{(--2.9767380742932996--0.43743242381802494*\x)/2.8949665690340676});
\draw [line width=1.pt,dash pattern=on 1pt off 2pt on 3pt off 4pt,domain=-3.3328037910523833:4.790630352259127] plot(\x,{(-0.06434015559752504-0.8570096793019588*\x)/-0.02709156021604148});
\draw [line width=1.pt,dash pattern=on 3pt off 3pt,domain=-3.3328037910523833:4.790630352259127] plot(\x,{(--1.6790264862969848-0.19851981163946997*\x)/0.6866775712265181});
\draw [line width=1.pt,domain=-3.3328037910523833:4.790630352259127] plot(\x,{(--4.9883830715538515--1.7245093507030378*\x)/2.042199994610631});
\draw [line width=1.pt,domain=-3.3328037910523833:4.790630352259127] plot(\x,{(-2.087427427447703--1.2870769268850129*\x)/-0.8527665744234366});
\begin{scriptsize}
\draw [fill=white] (-2.04,0.72) circle (2.5pt);
\draw[color=black] (-1.9,0.6)node {$p$};
\draw [fill=white] (0.8549665690340674,1.157432423818025) circle (2.5pt);
\draw[color=black] (.7,1)node {$q$};
\draw [fill=uuuuuu] (0.01951044368149565,2.992105070126037) circle (2.0pt);
\draw [fill=uuuuuu] (-0.00758111653454583,2.135095390824078) circle (2.0pt);
\draw [fill=uuuuuu] (-0.377379016256191,2.554246375614137) circle (2.0pt);
\draw [fill=uuuuuu] (0.309298554970327,2.355726563974667) circle (2.0pt);
\draw [fill=uuuuuu] (-0.04277485656649297,1.0217827717761603) circle (2.0pt);
\draw [fill=uuuuuu] (3.2187409229399466,1.5146011580092924) circle (2.0pt);
\draw [fill=uuuuuu] (0.0021999946106308274,2.4445093507030378) circle (2.0pt);
\draw [fill=uuuuuu] (-0.15984105596442819,2.3076758830446895) circle (2.0pt);
\draw [fill=uuuuuu] (0.13940057208450501,2.2374331086008934) circle (2.0pt);
\draw [fill=uuuuuu] (0.19480623768081337,2.607153199520421) circle (2.0pt);
\draw [fill=uuuuuu] (-0.2000943724477912,2.7498313697511456) circle (2.0pt);
\end{scriptsize}
\end{tikzpicture}
\caption{A supersolvable line arrangement with 2 modular points 
of equal multiplicity with possible added lines.}
\label{FigSS 3 3}
\end{figure}
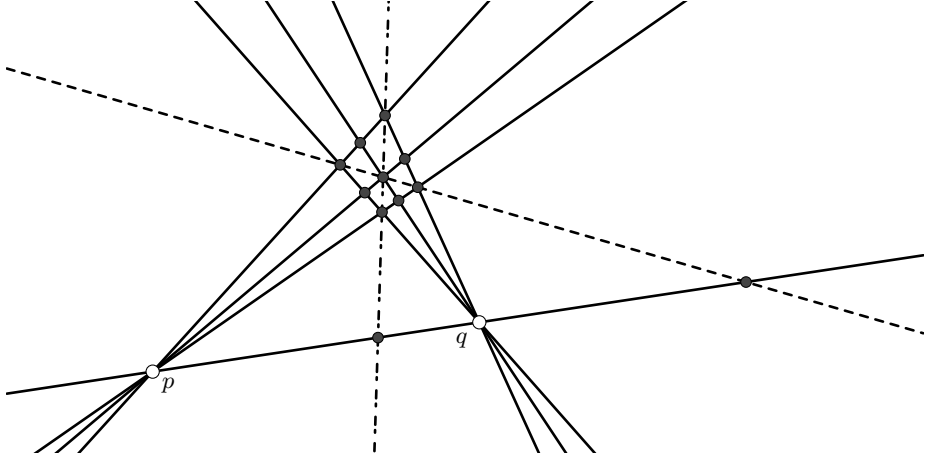

\section{Points of multiplicity 2 in supersolvable line arrangements}\label{Conjs}
\subsection{Questions and conjectures}
By Inequality \eqref{melchior}, every non-pencil real line arrangement has $t_2\geq 3$.
More generally, there is the still open Dirac-Motzkin Conjecture \cite{D}:

\begin{conjecture}\label{DMConj}
The inequality $t_2\geq \lfloor s/2\rfloor$ holds
for every non-pencil real line arrangement of $s$ lines.
\end{conjecture}

Things over $\mathbb C$ are more complicated. 
Four types of complex line arrangements with $t_2=0$ are currently known:
pencils of 3 or more lines; the lines defined by the linear factors of
$(x^n-y^n)(x^n-z^n)(y^n-z^n)$ for $n\geq 3$ (known as the Fermat arrangement, $\mathcal{F}_n$);
an arrangement due to F.\ Klein \cite{K} with 21 lines
and $t_k=0$ except for $t_3=28$ and $t_4=21$;
and an arrangement due to A.\ Wiman \cite{W} 
with 45 lines and $t_k=0$ except for $t_3=120$, $t_4=45$ and $t_5=36$
(see \cite{BDHHSS} for more information about the Klein and Wiman
arrangements).

We believe the following question is open.

\begin{question}
Are there any complex line arrangements with $t_2=0$ other than the four types listed above?
\end{question}

For the case of supersolvable line arrangements we pose the following conjecture.

\begin{conjecture}\label{conj1}
Every nontrivial complex supersolvable line arrangement has $t_2>0$.
\end{conjecture}

A much stronger conjecture was posed by \cite{AT}.

\begin{conjecture}\label{conj2}
Every non-pencil complex supersolvable line arrangement of $s$ lines has $t_2\geq s/2$.
\end{conjecture}

We have found all complex supersolvable line arrangements with at least 3 modular points,
and for these $t_2\geq s/2$ holds. Thus if the conjecture is false, then it must fail for
a line arrangement with either one or at most 2 modular points.

It is also interesting to ask:

\begin{question}\label{DMQuesOverC}
Which non-pencil complex line arrangements of $s$ lines fail to satisfy $t_2\geq \lfloor s/2\rfloor$?
\end{question}

Of course, as noted above, there are non-pencil line arrangements 
with $t_2=0$, and for these $t_2\geq \lfloor s/2\rfloor$ fails to hold.
Also, by adding or deleting lines from such line arrangements one can sometimes
get additional examples. For example, the line arrangement $\mathcal L$ with $s=3n$ lines defined
by the linear factors of $(x^n-y^n)(x^n-z^n)(y^n-z^n)$ has $t_2=0$; by
adding the line $x=0$, we get a line arrangement $\mathcal L'$ with $s=3n+1$ 
and $t_2=n$, so $t_2\geq \lfloor s/2\rfloor$ still fails. 
For another example, each line of the Klein arrangement of 21 lines contains four crossing points of multiplicity 4
and four of multiplicity 3. By removing one line we thus get an arrangement of $s=20$ lines
with $t_4=17$, $t_3=28$ and $t_2=4$, so here too $t_2\geq \lfloor s/2\rfloor$ fails.
But this leaves the question: are there any examples where $t_2\geq \lfloor s/2\rfloor$ fails to hold
which do not come in this way from the known examples with $t_2=0$?

If $\mathcal{L}$ is defined over $\mathbb{R}$, \cite{AT}
proves Conjecture \ref{conj2} over the reals (see \cite[Theorem 2.4]{AT}). A key step in their proof is 
\cite[Lemma 2.2]{AT}, a version of which we now state.
For the convenience of the reader we include a slightly simplified version of the proof from \cite{AT}.

\begin{lemma}\label{AT Lemma}
Let $p$ be a modular point of some multiplicity $m$ in a non-pencil real supersolvable line arrangement 
$\mathcal{L}$ containing $s$ lines. Then every line in $\mathcal{L}$ not containing $p$ contains 
a crossing point of multiplicity 2.
\end{lemma}

\begin{proof}
At left in Figure \ref{FigAT} we see the $m$ lines ($L_1,\dots,L_m$ enumerated from bottom to top)
through $p$ and some line $L$ not through $p$. To these we've added a dotted line below $L_1$,
and a dashed line above $L_m$. After a change of coordinates, the dotted line
becomes $y=0$, the dashed line becomes the line $z=0$ at infinity, $L$ becomes $x=0$ and $p$ becomes
the point $(1,0,0)$. Thus in the affine plane as shown at right in Figure \ref{FigAT}, 
the lines $L_i$ become horizontal lines and $L$ becomes vertical.

Let $p_i$ be the point of intersection of $L_i$ with $L$. Since $p$ is modular, every line
in $\mathcal{L}$ (other than $L$ itself) must intersect $L$ at one of the points $p_i$.
We ant to show that one of the points $p_i$ has multiplicity 2.
Suppose by way of contradiction that the multiplicity of $p_i$
is more than 2 for each $i$. Thus we can pick an additional line $H_i$ in $\mathcal{L}$ through $p_i$ for each $i$.
The slope of $H_i$ in the affine picture at right in Figure \ref{FigAT} is defined and not 0.

For each $i\neq $, the intersection of $H_i$ and $H_j$ must be on one of the lines $L_k$, since $p$ is modular.
If the slopes of $H_1$ and $H_m$ have the same sign, it is easy to see that they intersect either above
$L_m$ (if the slopes are both positive and $H_1$ has the larger slope, or if the slopes are both negative
and $H_1$ has the more negative slope) or below $L_1$ (if the slopes are both positive 
and $H_m$ has the larger slope, or if the slopes are both negative and $H_m$ has the more negative slope).

Thus in order for $p$ to be modular, $H_1$ and $H_m$ must have slopes of opposite sign.
This means as you go from $H_1$ to $H_2$ and on to $H_m$, there is a least $i$ such that
$H_i$ and $H_{i+1}$ have slopes of opposite sign. But this means that $H_i$ and $H_{i+1}$
intersect between $L_i$ and $L_{i+1}$ and hence that the point of intersection is not on
any of the horizontal lines $L_k$, contradicting modularity of $p$.
Thus at least one of the points $p_i$ must have multiplicity 2.
(For example, we could have $p_m$ have multiplicity 2 so there would be no $H_m$, 
and $H_1,\dots,H_{m_1}$ could all meet
at a point of $L_m$.)
\end{proof}

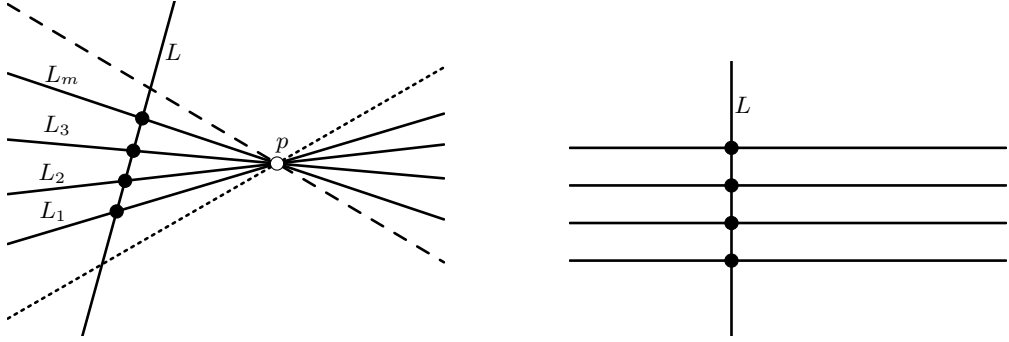
\begin{figure}
\hbox to\hsize{\hfil\hbox{\begin{tikzpicture}[line cap=round,line join=round,x=.5cm,y=.5cm]
\clip(-4.3,-2.58) rectangle (7.3,6.3);
\draw [line width=1.pt,domain=-4.3:7.3] plot(\x,{(--4.6708--1.22*\x)/4.08});
\draw [line width=1.pt,domain=-4.3:7.3] plot(\x,{(-6.7644-0.46*\x)/-4.04});
\draw [line width=1.pt,domain=-4.3:7.3] plot(\x,{(-8.6124--0.34*\x)/-3.82});
\draw [line width=1.pt,domain=-4.3:7.3] plot(\x,{(-11.8324--1.34*\x)/-4.});
\draw [line width=1.pt,dotted,domain=-4.3:7.3] plot(\x,{(--1.6096--2.64*\x)/4.58});
\draw [line width=1.pt,dash pattern=on 5pt off 5pt,domain=-4.3:7.3] plot(\x,{(-16.842--2.7*\x)/-4.56});
\draw [line width=1.pt,domain=-4.3:7.3] plot(\x,{(--1.2828--0.8*\x)/0.22});
\begin{scriptsize}
\draw [fill=white] (2.86,2.) circle (2.5pt);
\draw[color=black] (3.,2.49) node {$p$};
\draw[color=black] (-3.09,0.72) node {$L_1$};
\draw[color=black] (-3.11,1.74) node {$L_2$};
\draw[color=black] (-2.99,3.02) node {$L_3$};
\draw[color=black] (-2.85,4.36) node {$L_m$};
\draw[color=black] (.1,4.97) node {$L$};
\draw [fill=black] (-1.404142075043397,0.7249379089331018) circle (2.5pt);
%\draw[color=uuuuuu] (-1.26,1.05) node {$G$};
\draw [fill=black] (-1.18,1.54) circle (2.5pt);
%\draw[color=uuuuuu] (-1.04,1.87) node {$H$};
\draw [fill=black] (-0.96,2.34) circle (2.5pt);
%\draw[color=uuuuuu] (-0.82,2.67) node {$I$};
\draw [fill=black] (-0.7233810232345198,3.200432642783564) circle (2.5pt);
%\draw[color=uuuuuu] (-0.58,3.53) node {$J$};
\end{scriptsize}
\end{tikzpicture}}\hfil
\hbox{\definecolor{xdxdff}{rgb}{0.49019607843137253,0.49019607843137253,1.}
\definecolor{uuuuuu}{rgb}{0.26666666666666666,0.26666666666666666,0.26666666666666666}
\definecolor{cqcqcq}{rgb}{0.7529411764705882,0.7529411764705882,0.7529411764705882}
\begin{tikzpicture}[line cap=round,line join=round,x=.5cm,y=.5cm]
\clip(-4.3,-1) rectangle (7.3,6.3);
\draw [line width=1.pt] (0.,-3.08) -- (0.,6.3);
\draw [line width=1.pt,domain=-4.3:7.3] plot(\x,{(-9.-0.*\x)/-3.});
\draw [line width=1.pt,domain=-4.3:7.3] plot(\x,{(-6.-0.*\x)/-3.});
\draw [line width=1.pt,domain=-4.3:7.3] plot(\x,{(-3.-0.*\x)/-3.});
\draw [line width=1.pt,domain=-4.3:7.3] plot(\x,{(-12.-0.*\x)/-3.});
\begin{scriptsize}
%\draw [fill=uuuuuu] (0.,0.) circle (2.0pt);
%\draw[color=uuuuuu] (0.14,0.33) node {$A$};
\draw [fill=black] (0.,3.) circle (2.5pt);
%\draw[color=xdxdff] (0.14,3.37) node {$B$};
\draw[color=black] (0.3,5.15) node {$L$};
%\draw[color=black] (-4.14,2.85) node {$g$};
\draw [fill=black] (0.,2.) circle (2.5pt);
%\draw[color=xdxdff] (0.14,2.37) node {$C$};
%\draw[color=black] (-4.14,1.85) node {$h$};
\draw [fill=black] (0.,1.) circle (2.5pt);
%\draw[color=xdxdff] (0.14,1.37) node {$D$};
%\draw[color=black] (-4.14,1.33) node {$i$};
\draw [fill=black] (0.,4.) circle (2.5pt);
%\draw[color=xdxdff] (0.14,4.37) node {$E$};
%\draw[color=black] (-4.14,3.85) node {$j$};
\end{scriptsize}
\end{tikzpicture}}\hfil}
\caption{At left, a modular point $p$ of multiplicity $m$ in a real
supersolvable line arrangement $\mathcal{L}$ 
and a line $L$ in $\mathcal{L}=\{L_1,\dots,L_m\}$ not through $p$, and 
at right an affine version of the same arrangement after an appropriate change of coordinates
moving the dashed line to infinity.}
\label{FigAT}
\end{figure}

We now state and give a simplified proof of a slightly strengthened version of \cite[Theorem 2.4]{AT}.

\begin{theorem}
Let $\mathcal{L}$ be a real non-pencil supersolvable line arrangement
containing $s$ lines. Let $p$ be any modular point of $\mathcal{L}$ and let $m$ be the multiplicity of $p$.
Then $t_2 \ge \max\{s-m,m\}\geq s/2$. 
\end{theorem}

\begin{proof}
By Lemma \ref{AT Lemma}, each of the $s-m$ lines in $\mathcal{L}$
not through $p$ contain a point of multiplicity 2.
These points are all distinct since if two different lines not through $p$ shared
a point of multiplicity 2, no other lines in $\mathcal{L}$ could contain that point,
hence no line through $p$ could contain the point, contradicting
modularity of $p$. Thus $t_2\geq s-m$.
On the other hand, by Inequality \eqref{melchior} we have
$t_2\geq 3 + (m-3)t_m\geq 3+(m-3)=m$.
\end{proof}

The preceding result prompts the following question:

\begin{question}\label{QuesOverC}
Does every non-pencil supersolvable complex line arrangement of $s$ lines with a modular
point of multiplicity $m$ satisfy $t_2 \ge \max\{s-m,m\}$?
\end{question}

In the direction of Conjecture \ref{conj1}, we prove the following result. 

\begin{theorem} \label{m=3,4}
Let $\mathcal{L}=\{L_1,\ldots,L_s\}$ be a nontrivial complex line arrangement
(i.e., not a pencil or near pencil).
Assume that every crossing point of $\mathcal{L}$ has
multiplicity equal to 3 or 4. Then the line arrangement $\mathcal{L}$ is not supersolvable.
\end{theorem}

\begin{proof}
%We  have the following inequality due to Hirzebruch \cite{Hir83} which
%holds for any nontrivial arrangement:
%$$t_2+\frac{3}{4}t_3\geq s + \sum_{k>4}t_k(k-4).$$

Since $\mathcal{L}$ is not a pencil or a near pencil by
hypothesis, we can apply Inequality \eqref{hirzebruch}.  
In our case, it takes the form:
$\frac{3}{4}t_3\geq s$. 

%We also have the following combinatorial identity which is valid for
%any line arrangement: $$\binom{s}{2} = \sum_{k\geq2}\binom{k}{2}t_k.$$

By \eqref{comb-id}, we have $s(s-1) = 6t_3+12t_4$. 

Suppose that $\mathcal{L}$ is supersolvable. Then, by \eqref{at}, we have 
$t_2\geq 2n-m(s-m)-2$, where $n$ is the total number of crossings and
$m$ is the maximum $k$ such that $t_k>0$.
In our case, this gives $0 \ge 2(t_3+t_4) -m(s-m)-2$, where $m = 3$ or $m=4$. 

First we assume $m=4$ and obtain a contradiction. We have $2(s-4) +1 \ge t_3+t_4$. This implies 
$12(s-4) +6 \ge 6(t_3+t_4) \ge 8s+6t_4$. The last inequality follows
from the Hirzebruch inequality. So we get $6t_4 + 12(s-4)+6 \ge 6t_3+12t_4
= s(s-1)$, where the last equality follows from \eqref{comb-id}.

This, in turn, gives, $12(s-4) +6   \ge 6t_4+8s \ge s(s-1)-12(s-4)-6+8s$. Looking at the
first and third terms in this and rearranging terms, we get
$s^2-17s+84 \le 0$. But since this quadratic in $s$ has positive leading
coefficient and negative discriminant, $s^2-17s+84 > 0$ for
every $s$, giving us the desired contradiction. 

The calculation is similar if $m=3$. By 
\eqref{at}, we get $3(s-3)+2 \ge 2t_3$. Using the 
Hirzebruch inequality \eqref{hirzebruch}, we get $9(s-3)+6 \ge 6t_3 \ge 8s$. This forces
$s\ge 21$. On the other hand, $s(s-1) = 6t_3$ by \eqref{comb-id}. Hence we
obtain $9(s-3)+6 \ge 6t_3 = s(s-1)$, or equivalently, $(s-3)(s-7) \le
0$. So $3 \le s \le 7$. This is not possible. 
\end{proof}

\begin{example}
We do not know many nontrivial examples of complex line arrangements
where every crossing point has multiplicity 3 or 4.
We get two examples by taking the lines defined by the linear factors of
$(x^n-y^n)(x^n-z^n)(y^n-z^n)$ for $n=3$ and $n=4$.
The only other example we know is the one due to Klein \cite{K},
having 21 lines with $t_k=0$ except for $t_3=28$ and $t_4=21$.
\end{example}

\begin{remark}
Let $\mathcal{L}$ be a supersolvable 
line arrangement in ${\mathbb P}_{\mathbb C}^2$ of $s$ lines. Assume that $t_2 =
0$ and let $m$ be the largest multiplicity of a crossing point of
$\mathcal{L}$. By Theorem \ref{m=3,4}, we must have $m \ge 5$. 
Arguing as in the proof of Theorem \ref{m=3,4}, we can show that, for
a fixed $m$, $s$
must be in a specific set of positive integers. For
example, if $m=5$, then our arguments show that $10 \le s \le
13$. Similarly, if $m=100$, then $103 \le s \le 7412$.  But we do not
know of any supersolvable line arrangements in ${\mathbb P}_{\mathbb
  C}^2$ for which $m=100$ and $s \ge 7413$. If there is such a line
arrangement, then our argument shows that it must contain a point
of multiplicity 2. 
\end{remark}

\section{Applications to unexpectedness}\label{appls}

One of the most interesting applications of line arrangements in
$\mathbb{P}^2$ is to finding unexpected curves. More specifically, given a
line arrangement in $\mathbb{P}^2$ one considers the dual arrangement
of points. The question then is whether these points admit an
\textit{unexpected curve}. For more details, see \cite{CHMN}.

The existence of unexpected curves depends on some properties of the
line arrangement. If the arrangement is supersolvable, then
\cite[Theorem 3.17]{DMO} proves that there is
an unexpected curve through the dual points if and only if $s > 2m$,
where $s$ is the number of lines and $m$ is the maximum multiplicity
of a crossing point. We now use this characterization to determine which
supersolvable arrangements in the classification of Section \ref{classification} admit
unexpected curves.  

\subsection{Real line arrangements admitting unexpected curves} First, let us consider a real supersolvable line arrangement
$\mathcal{L}$. 

If $\mathcal{L}$ has exactly one modular point, then the only
arrangement we know which satisfies the condition $s>2m$ is given by
considering a regular $n$-gon for even $n$ and adding the line at infinity. For
more details, see \cite[Theorem 3.15]{DMO}. 

If $\mathcal{L}$ has exactly two modular points, then the only
arrangement which admits an unexpected curve is given by the
following. Let $m\geq 6$ be even and consider an arrangement of $m$ horizontal and $m$ vertical
lines, along with the line at infinity. This is supersolvable with the two modular points of
multiplicity $m+1$ at infinity where the horizontal and vertical lines
meet the line at infinity. Since there are only $s = 2m+1$ lines, this arrangement does not
admit an unexpected curve. But we can add the two diagonals 
(as in Figure \ref{FigSS 3 3}, which shows the case of $m=4$, but in that case
there are three modular points) to this
arrangement without changing the maximum multiplicity while preserving
supersolvability. 
Now the condition $s=2m+3 > 2(m+1)$ is satisfied
and hence the new arrangement admits an
unexpected curve. This arrangement is a
special type of \textit{tic-tac-toe} arrangement described in
\cite[Theorem 3.19]{DMO}. 
The multiplicities of the two modular points (or three when $m=4$) in this tic-tac-toe
arrangement  are equal. There are no other supersolvable arrangements
with exactly two modular points which
admit unexpected curves. 

The only other real supersolvable line arrangement admitting an
unexpected curve is the Fermat arrangement for $n=2$ with three
coordinate axes added. More precisely, this arrangement is defined by 
$xyz(x^2-y^2)(x^2-z^2)(y^2-z^2)=0$.
This has 9 lines and three modular points of multiplicity 4 each
(it is displayed in Figure \ref{FigSS 3 3}). 

In summary, except for possibly more supersolvable arrangements with a
unique modular point, the only real supersolvable line arrangements
which admit an expected curve are listed above. We ask the following
question. 

\begin{question}
Are there any other real supersolvable line arrangements (other than
the one coming from a regular $n$-gon) with exactly
one modular point whose dual points admit an unexpected curve? 
\end{question}

\subsection{Complex line arrangements admitting unexpected curves}
We now consider complex line arrangements. The only examples known to
us of supersolvable arrangements which admit unexpected curves are
obtained by adding two or three coordinate axes to the Fermat
arrangement $\mathcal{F}_n$. In other words, we 
are considering the complex line arrangement given by
$xy(x^n-y^n)(x^n-z^n)(y^n-z^n)$, or $xyz(x^n-y^n)(x^n-z^n)(y^n-z^n)=0$.
 
This has $s=3n+\epsilon$ lines, where $\epsilon=2$ or
$3$ and maximum multiplicity $m=n+2$. Hence the condition $s>2m$ is
satisfied for $\varepsilon =2, n\ge 3$ or $\varepsilon =3, n\ge
2$. In the first case, there is a unique modular point and in the
second case, there are three modular points. 

We end with the following question. 

\begin{question}
Are there any other complex supersolvable line arrangements (different
from the arrangements coming from the Fermat arrangement described above)
whose dual points admit an unexpected curve? 
\end{question}

\end{document}